\documentclass[10pt]{amsart}
\usepackage{verbatim}
\usepackage{amsmath,amsfonts,amscd,amssymb,epsfig,euscript,bm}
\usepackage{amsxtra,mathrsfs}
\usepackage{pdfsync}
\newtheorem{theorem}{Theorem}[section]
\newtheorem{proposition}{Proposition}[section]
\newtheorem{lemma}{Lemma}[section]

\theoremstyle{definition}
{}

\theoremstyle{remark} 
\newtheorem{remark}{Remark}[section]

\newcommand{\hth}{\hat{\theta}}

\newcommand{\spbp}{\sqrt{-1}\partial \bar{\partial}}

\numberwithin{equation}{section}
\usepackage{hyperref}
\hypersetup{colorlinks,citecolor=blue,plainpages=false,hypertexnames=false}
\begin{document}
\title[The deformed Hermitian Yang-Mills equation on three-folds]{The deformed Hermitian Yang-Mills equation on three-folds}
\author{Vamsi Pritham Pingali}
\address{Department of Mathematics, Indian Institute of Science, Bangalore, India - 560012}
\email{vamsipingali@iisc.ac.in}
\begin{abstract} 
We prove an existence result for the deformed Hermitian Yang-Mills equation for the full admissible range of the phase parameter, i.e., $\hat{\theta} \in (\frac{\pi}{2},\frac{3\pi}{2})$, on compact complex three-folds conditioned on a necessary subsolution condition. Our proof hinges on a delicate analysis of a new continuity path obtained by rewriting the equation as a generalised Monge-Amp\`ere equation with mixed sign coefficients. 
\end{abstract}
\maketitle
\section{Introduction}\label{Introsec}
\indent Motivated from mirror symmetry considerations, Jacob and Yau \cite{jacobyau} introduced the deformed Hermitian Yang-Mills equation for a metric $h$ on a holomorphic line bundle $L$ over a compact K\"ahler manifold $(X,\omega)$ of complex dimension $n$ :
\begin{gather}
\mathrm{Im}((\omega-\Theta)^n) =\tan(\hat{\theta}) \mathrm{Re}((\omega-\Theta)^n),
\label{gendhymequation}
\end{gather}  
where $\Theta$ is the curvature of the Chern connection of $h$ and $\hat{\theta}$ is a constant parameter (called the ``phase angle") satisfying the condition that the (topologically invariant) integrals of both sides of \ref{gendhymequation} are equal. This equation may be rewritten as follows \cite{jacobyau}. 
\begin{gather}
\hat{\theta} = \theta (\Theta),
\label{phaserewriting}
\end{gather}
where for any real $(1,1)$-form $\alpha$ whose associated endomorphism $A^i _j =\alpha_{j\bar{k}}\omega^{i\bar{k}}$ has eigenvalues $a_i$,
\begin{gather}
\theta(\alpha) := \sum_i \arctan(a_i).
\label{defofphase}
\end{gather}
In \ref{defofphase} the $\arctan$s lie in $(-\frac{\pi}{2},\frac{\pi}{2})$.\\
\indent In \cite{collinsjacobyau} the following existence result was proved for this equation.
\begin{theorem}[Collins-Jacob-Yau \cite{collinsjacobyau}]
Fix $\Theta_0 \in 2\pi c_1(L)$ and suppose the topological constant $\hth$ satisfies the critical phase condition $$\hth > \frac{(n-2)\pi}{2}.$$ Furthermore, suppose that there exists $\chi := \Theta_0 + \spbp u$ defining a subsolution. Assume that 
\begin{gather}
\theta(\chi)>\frac{(n-2)\pi}{2}.
\label{collinsphasecondition}
\end{gather}
 Then there exists a unique smooth Chern curvature form $\Theta \in 2\pi c_1(L)$ solving the deformed Hermitian-Yang-Mills equation \ref{gendhymequation}.
\label{collinsjacobyauresult}
\end{theorem}
\begin{remark}
The condition on $\chi$ is implied by $$\hth > \left(n-2+\frac{2}{n} \right )\frac{\pi}{2},$$ and is hence not vacuous. It is expected that such a hypothesis on $\chi$ can be removed \cite{collinsjacobyau}. Indeed, in the case of $n=2$, this expectation was met in \cite{jacobyau} by rewriting \ref{gendhymequation} as a Monge-Amp\`ere equation. 
\end{remark}
The subsolution condition alluded to in Theorem \ref{collinsjacobyauresult} can be written as follows. (Proposition 8.1 in \cite{collinsjacobyau}.)
\begin{gather}
\Omega>0, \ n\Omega^{n-1}-\displaystyle \sum_k b_k k\Omega^{k-1} \omega^{n-k}>0,
\label{subsolution}
\end{gather}
where $\Omega=\sqrt{-1}\Theta -\omega\tan(\hat{\theta})$ and $b_k=\csc^{n-k}(\hth)(-1)^{n-k+1}{n \choose k}\sin((n-k-1)\hth)$ if $n$ is even and when $n$ is odd, $b_{2j}= \sec^{n+1-2j}(\hth)(-1)^{\frac{n+2j+1}{2}}{n \choose 2j}\sin((n-2j-1)\hth)$ and $b_{2j+1}= \sec^{n-2j}(\hth)(-1)^{\frac{n+2j+1}{2}}{n \choose 2j+1}\cos((n-2j-2)\hth)$.\\
\indent A considerable amount of work has been done on this equation \cite{collinsjacobyau, collinspicard, gchen, Pingen, pindhym, wang, smocyzk, jacob, collinsxieyau, han, mawang, yama, hanyama, collinsyau}. However, the hypothesis on the range of $\hat{\theta}$ has not been weakened beyond $n=2$. In this paper, we prove an existence result for the full admissible range of $\hat{\theta}$ in the case of three dimensions. Our main theorem is as follows.
\begin{theorem}
Let $(X,\omega)$ be a compact K\"ahler three-fold with a holomorphic line bundle $L$ over it. Assume that there exists a Hermitian metric $h_0$ on $L$ such that the corresponding form $\Omega = \sqrt{-1} \Theta_{h_0} -\omega \tan(\hth)$ satsifies the subsolution condition \ref{subsolution}. Then there exists a smooth solution to the deformed Hermitian Yang-Mills equation \ref{gendhymequation} if the phase angle $\hth$ satisfies $$\hth \in \left(\frac{\pi}{2}, \frac{3\pi}{2} \right).$$
\label{maintheorem}
\end{theorem}
The proof of theorem \ref{maintheorem} involves rewriting Equation \ref{gendhymequation} as the following generalised Monge-Amp\`ere equation (akin to the proof of Theorem 1.4 in \cite{pindhym}), to be solved for a smooth function $\phi$ such that $\Omega_{\phi}=\Omega+\spbp \phi >0$ and $3\Omega_{\phi}^2-3\omega^2\sec^2(\hth)>0$.
\begin{gather}
\Omega_{\phi}^3 = 3\omega^2 \Omega_{\phi}\sec^2(\hat{\theta})+2\omega^3 \tan(\hat{\theta}) \sec^2(\hat{\theta}).
\label{genma}
\end{gather}
To solve \ref{genma} we consider two cases. When $\tan(\hat{\theta})\geq 0$, the result follows from Theorem 1.2 of  \cite{pindhym}. Assume that $\tan(\hat{\theta})<0$. In this case, the method of continuity based on a path (Equation \ref{continuitypath} in Section \ref{Opennesssec})depending on a parameter $0\leq t \leq 1$ is used. At $t=0$, it turns out that existence holds because the resulting equation is a standard Hessian equation originally studied by Fang-Lai-Ma \cite{FangLaiMa}. In Section \ref{Opennesssec}, ellipticity is proved to be preserved along this continuity path, and the set of $t$ for which there is a solution is proved to be open. For closedness, we need \emph{a priori} estimates as usual. They are proved in steps in Section \ref{uniformandgradientestimatesec} and Section \ref{concludingsec}. The technical challenge in this otherwise standard procedure is that Equation \ref{genma} has mixed signs. Equations akin to it had been studied prior to this work \cite{gchen, zheng} but the ``bad" sign piece had been assumed to be quantifiably small. In this paper, it fortuitously turns out that despite $\tan(\hat{\theta})$ potentially being quite negative, the desired estimates can still be proved, albeit in a delicate manner. \\
\indent We expect that a similar continuity path would work for higher dimensions and be expected to aid in proving a result similar to Theorem 1.7 of \cite{gchen}. However, the complexity involved is formidable and we plan on addressing the same in future work.\\

\emph{Acknowledgements} : The author thanks Tristan Collins for clarifying some aspects of \cite{collinsjacobyau}. The author is also grateful to the anonymous referee for constructive feedback. The author is grateful to Chao-Ming Lin for pointing out an error in an earlier draft.

\section{Method of continuity - Openness}\label{Opennesssec}
\indent In the rest of the paper, we assume that $\hat{\theta} \in (\frac{\pi}{2},\pi)$, i.e., $\tan(\hat{\theta})<0$. To solve Equation \ref{genma} for this range of $\hat{\theta}$ assuming the existence of a metric $h_0$ whose curvature $\Theta_0$ satisfies $\Omega=\sqrt{-1}F_0 -\omega\tan(\hat{\theta})>0$ and $3\Omega^2-3\omega^2 \sec^2(\hat{\theta})>0$, we use the following continuity path depending on a parameter $0\leq t \leq 1$.
\begin{gather}
\Omega_{\phi}^3 = 3c_t\omega^2 \Omega_{\phi}\sec^2(\hat{\theta})+2t\omega^3 \tan(\hat{\theta}) \sec^2(\hat{\theta}),
\label{continuitypath}
\end{gather}
where 
\begin{gather}
c_t = \displaystyle \frac{\int \Omega^3-2t\tan(\hth) \sec^2(\hth)\int \omega^3}{\sec^2(\hth)\int 3\omega^2 \Omega}.
\label{defofct}
\end{gather}
All the properties of this constant that we shall need later on are summarised in the following lemma.
\begin{lemma}
$c_t$ satisfies the following properties.
\begin{enumerate}
\item $\frac{1}{3} < c_t \leq 1$.
\item $c_t^3 >t^2 \sin^2(\hth)$.
\end{enumerate}
\label{usefulpropsofct}
\end{lemma}
\begin{proof}
Before proving the two statements, we recall that 
\begin{gather}
\displaystyle \int \Omega^3 = \sec^2(\hth)\int 3\omega^2 \Omega + 2\tan(\hth) \sec^2(\hth) \int \omega^3. \label{defofthetaagain}
\end{gather}
\begin{enumerate}
\item Substituting \ref{defofthetaagain} in \ref{defofct} we see that $$c_t=1+(1-t)2\tan(\hth)\frac{\int\omega^3}{\int 3\omega^2 \Omega} \leq 1$$ because $\tan(\hth)<0$. Since $\Omega^2 > \omega^2 \sec^2(\hth)$, we see that $\displaystyle \int \Omega^3 > \int  \omega^2 \Omega \sec^2(\hth)$, i.e., $\alpha=\frac{\int \Omega^3}{\sec^2(\hth)\int 3\omega^2 \Omega}$ satisfies $\frac{1}{3}<\alpha < 1$. Using \ref{defofthetaagain} and \ref{defofct} again, we see that
$$c_t = \displaystyle (1-t)\alpha +t >\frac{1-t}{3}+t\geq \frac{1}{3}.$$
\item We find a lower bound on $f(t)=\frac{c_t^{3/2}}{t}$ when $0<t\leq 1$. The critical point is easily seen to lie at $t=\frac{2\alpha}{1-\alpha} >1$. Hence, the minimum is attained at $t=1$. Therefore $f(t)\geq 1>\vert \sin(\hth) \vert$ and hence $c_t^3 > t^2 \sin^2(\hth)$.
\end{enumerate}
\end{proof}
\indent The strategy of the proof of theorem \ref{maintheorem} is the method of continuity, i.e., to prove that the set $S$ of $t\in [0,1]$ such that Equation \ref{continuitypath} has a smooth solution $\phi$ satisfying $\displaystyle \int \phi \omega^n =0$ is non-empty, open, and closed. At $t=0$ the equation reduces to a Hessian equation studied originally by Fang-Lai-Ma. It has a solution if and only if the cone condition $\Omega^2-\omega^2 \sec^2(\hth) c_0 >0$ is satisfied. (Theorem 1.1 in \cite{FangLaiMa}.) Since $\Omega^2-\omega^2 \sec^2(\hth)>0$ by assumption and $c_0<1$, the cone condition is met and $S$ is therefore not empty. \\
\indent To prove openness, the usual method is the implicit function theorem on Banach manifolds. All that is needed is to prove that the ellipticity of the linearised equation is preserved by the continuity path - a fact asserted by the following lemma.
\begin{lemma}
The linearisation of Equation \ref{continuitypath} given by $$Lu =(3\Omega_{\phi}^2-3c_t\omega^2 \sec^2(\hth))\spbp u$$ is an elliptic operator for all $t$ lying in the connected component of $S$ containing $0$. Moreover, $\Omega_{\phi}>0$ for all such $t$.
\label{ellipticitypreservation}
\end{lemma}
\begin{proof}
 It suffices to prove that for all such $t$, the conditions $*_t$ given by 
\begin{gather}
3\Omega_{\phi}^2-3c_t\omega^2 \sec^2(\hth)>0, \ and \nonumber 
\\ \Omega_{\phi}>0 
\end{gather} hold. Obviously $*_0$ hold. The conditions $*_t$ are clearly open. It suffices to prove their closedness, i.e., if $t_n \rightarrow t \in S$ and $*_{t_n}$ hold, then so do $*_t$. Suppose not. That is, let $t$ be the first time where $*_t$ fail. At such a $t$, there exists a point $p$ such that one of the conditions in $*_t$ become degenerate. Let us choose normal coordinates for $\omega$ near $p$ and also diagonalise $\Omega_{\phi}(p)=\displaystyle \sum_i \lambda_i \sqrt{-1}dz^i \wedge d\bar{z}^i$, where $\lambda_3\leq \lambda_2 \leq \lambda_1$. We see that at $p$,
\begin{gather}
\lambda_1 \lambda _2 \lambda _3 = c_t \displaystyle \sum_i \lambda_i \sec^2(\hth) +2t\tan(\hth)\sec^2(\hth) \nonumber \\
\Rightarrow \lambda_1 [\lambda_2 \lambda _3 - c_t \sec^2(\hth)] = c_t(\lambda_2+\lambda_3)\sec^2(\hth)+2t\tan(\hth) \sec^2(\hth), 
\label{Atpopenness}
\end{gather}
and
\begin{gather}
\lambda_i \geq 0, \ \forall \ i\nonumber \\
\lambda_i \lambda_j \geq c_t \sec^2 (\hth), \ \forall \ i\neq j \nonumber \\
 c_t(\lambda_i+\lambda_j)+2t\tan(\hth)  \geq 0 \ \forall \ i\neq j.
\label{conditionsholdingatp}
\end{gather}
The second condition in \ref{conditionsholdingatp} shows that $\lambda_i >0 \ \forall \ i$. The least value of the third expression is bounded as follows.
\begin{gather}
c_t(\lambda_2+\lambda_3)+2t\tan(\hth) \geq 2c_t\sqrt{\lambda_2 \lambda_3} +2t\tan(\hth) \geq 2c_t^{3/2} \vert \sec(\hth) \vert+2t\tan(\hth)>0,
\label{leastvalueofexp}
\end{gather}
where the last inequality follows from Lemma \ref{usefulpropsofct}. Therefore, the only possible way to reach the boundary of the conditions in \ref{conditionsholdingatp} is through
$$\lambda_2 \lambda_3 = c_t\sec^2(\hth).$$
However, in this case, \ref{Atpopenness} and \ref{leastvalueofexp} produce a contradiction. Thus, $*_t$ also holds, completing the proof of the lemma.
\end{proof}
Given that the linearisation of \ref{continuitypath} is elliptic, openness is standard and is hence omitted.
\section{$C^0$ and Laplacian estimates}\label{uniformandgradientestimatesec}
\indent To prove the closedness of $S$, one needs \emph{a priori} estimates. As usual, we prove them in stages - a $C^0$ estimate which leads to a Laplacian estimate and further to higher order estimates. To prove the $C^0$ estimate we follow the proof of Proposition 11 in \cite{gabor} which is in turn based on a method of Blocki \cite{Blocki,Blocki2}. 
\begin{proposition}
Along the continuity path \ref{continuitypath}, $\Vert \phi \Vert_{C^0} \leq C_0$ where $C_0$ is independent of $t$.
\label{coestimate}
\end{proposition}
\begin{proof}
By adding a constant, we may renormalise $\phi$ to satisfy $\sup \phi =0$. Therefore, we only need to obtain a lower bound for $L=\inf \phi$. A standard Green's function argument shows that we have a uniform $\Vert \phi \Vert_{L^1}$ bound. The weak Harnack inequality then implies an $L^p$ bound for some $p>2$ as in \cite{gabor}.\\
\indent Choose coordinates $z$ satisfying $\vert z \vert <1$ such that $L$ is attained at the origin. Let $v=\phi+\epsilon \vert z \vert^2$ for some small $\epsilon>0$ to be chosen as we go along. An Alexandroff-Bakelman-Pucci type maximum principle (Proposition 10 in \cite{gabor}) shows that the following inequality holds.
\begin{gather}
c_0 \epsilon ^{2n} \leq \int_{P} \det(D^2 v),
\end{gather}
where $P$ is the ``lower contact set" (defined in Proposition 10 in \cite{gabor}) where $v$ is convex and $\vert Dv \vert \leq \frac{\epsilon}{2}$. Since $\det(D^2 v) \leq 2^{2n} \det(v_{i\bar{j}})^2$ on $P$, if we manage to prove that  $\vert v_{i\bar{j}} \vert \leq C$ on $P$, then we are done. Indeed, on $P$, $L<v(x)<L+\frac{\epsilon}{2}$ and hence $$\frac{\Vert v \Vert_{L^p}^p}{\vert L+\frac{\epsilon}{2} \vert^p}\geq vol(P)\geq C\epsilon^{2n}.$$ This inequality produces a contradiction if $\vert L \vert$ is large because $\Vert v \Vert_{L^p}\leq \Vert \phi \vert_{L^p}+C$. \\
\indent We now have the following estimate.
\begin{lemma}
On $P$, $\vert \phi_{i\bar{j}} \vert \leq C$ for some $C$ independent of $t$.
\label{c0useful}
\end{lemma} 
\begin{proof}
At a point $p \in P$, choose holomorphic normal coordinates for $\omega$ such that $\Omega_{\phi}(p)=\displaystyle \sum_i \lambda_i \sqrt{-1} dz^i \wedge d\bar{z}^i$. On $P$, since $D^2 v \geq 0$, we see that $\phi_{i\bar{j}} \geq -\epsilon \delta_{i\bar{j}}$ everywhere and in particular, at $p$. Therefore, $\lambda_i\geq \Omega_{i\bar{i}}(p)-\epsilon$. Choose $\epsilon$ so small that $\Omega \geq 2\epsilon \omega$ throughout $P$. Moreover, since $\Omega^2 > \omega^2 \sec^2(\hth)$ by assumption, if $\epsilon$ is small enough, throughout $P$ we have $\Omega_{\phi}^2 -\omega^2 \sec^2(\hth) \geq \frac{\Omega^2-\omega^2 \sec^2(\hth)}{2}$. Assume that $\lambda_3 \leq \lambda_2 \leq \lambda_1$. Note that $\lambda_2 \geq \lambda_3 \geq \epsilon$. We claim that $\lambda_1 \leq C$ for some $C$. Suppose not, i.e., there exists a sequence of times $t_n\rightarrow T$ such that $\lambda_{1,n}(p_n)\rightarrow \infty$. We abuse notation by writing $\lambda_1\rightarrow \infty$. \\
 Using \ref{Atpopenness} we get the following.
\begin{gather}
\lambda_3 = \sec^2(\hth)\frac{c_t(\lambda_2+\lambda_1)+2t\tan(\hth)}{\lambda_2 \lambda_1-c_t\sec^2(\hth)}.
\end{gather}
If $\lambda_2\rightarrow \infty$, then $\lambda_3\rightarrow 0$, which is not possible. Hence, $\lambda_2$ is bounded. Therefore, $\lambda_2\lambda_3\rightarrow c_t \sec^2(\hth)$, which is a contradiction because $\lambda_3 \lambda_2 -c_t \sec^2(\hth) \geq \frac{(\Omega^2-\omega^2 \sec^2(\hth))_{1\bar{1}2\bar{2}}}{2} >0$. Therefore, $\vert \phi_{i\bar{j}} \vert \leq C$ for all $p\in P$.
\end{proof}
Therefore, a uniform  $C^0$ estimate holds.
\end{proof}
Before proving further estimates, we define the expressions  
\begin{gather}
F(A) = \frac{3c_t\omega^2 \Omega_{\phi}}{\Omega_{\phi}^3}+\frac{2t\omega^3\tan(\hth)}{\Omega_{\phi}^3} \nonumber\\
\tilde{F}(A)=F(A) \sec^2 (\hth)\label{defofF},
\end{gather}
where $A^{i}_{j}=\omega^{i\bar{k}}(\Omega_{\phi})_{j\bar{k}}$. So Equation \ref{continuitypath} can be written as
\begin{gather}
\tilde{F}(A)=1.
\label{continuitypathintermsofF}
\end{gather}

\indent We first prove a useful convexity property of $\tilde{F}$ when restricted to a certain subset.
\begin{lemma}
Let $F,\tilde{F}$ be as above. Let $A=diag(\lambda_1, \lambda_2, \lambda_3)$ be a positive-definite diagonal matrix, and $B$ be a $3\times 3$ Hermitian matrix satisfying
\begin{gather}
\left (c_t-\frac{c_t \sum_i \lambda_i +2t\tan(\hth)}{\lambda_{\mu}} \right )B_{\mu \bar{\mu}}=0,
\label{conditiononB}
\end{gather}
where $\lambda_i$ satisfy 
\begin{gather}
\tilde{F}(A) = 1, \nonumber \\
\lambda_i \lambda_j > c_t \sec^2 (\hth) \nonumber \\
(\lambda_i+\lambda_j)c_t +2t\tan(\hth) >0. 
\label{lambdasatisfy}
\end{gather}
Then the following convexity property holds at $A$.
\begin{gather}
\frac{\partial^2 F}{\partial A_{\mu \bar{\nu}} \partial A_{\alpha \bar{\beta}}} B_{\mu \bar{\nu}}B_{\alpha \bar{\beta}}\geq \frac{c_t \sum_i \lambda_i +2t\tan(\hth)}{\det(A)} \sum_{\alpha \neq \mu}\frac{\vert B_{\mu \bar{\alpha}} \vert^2}{\lambda_{\mu}\lambda_{\alpha}}.
\label{convexeq}
\end{gather}
\label{convexitylem}
\end{lemma}
\begin{proof}
Since $$\frac{\partial F}{\partial A_{\mu \bar{\nu}}}=\frac{c_t \delta_{\mu \bar{\nu}}-(c_t \mathrm{tr}(A) +2t\tan(\hth))(A^{-1})_{\nu \bar{\mu}}}{\det(A)},$$ differentiating again we get
\begin{align}
\frac{\partial^2 F}{\partial A_{\mu \bar{\nu}}\partial A_{\alpha \bar{\beta}}}&= \frac{1}{\det(A)} \Bigg ( -\frac{c_t \delta_{\alpha \bar{\beta}}\delta_{\nu \bar{\mu}}}{\lambda_{\mu}}+\frac{(c_t \mathrm{tr}(A)+2t\tan(\hth))\delta_{\nu \bar{\alpha}}\delta_{\beta \bar{\mu}}}{\lambda_{\alpha}\lambda_{\mu}} \nonumber \\
&-\left(c_t \delta_{\mu \bar{\nu}}-(c_t \mathrm{tr}(A) +2t\tan(\hth))\frac{\delta_{\nu \bar{\mu}}}{\lambda_{\mu}}\right)\frac{\delta_{\beta \bar{\alpha}}}{\lambda_{\alpha}}\Bigg ). 
\end{align} 
Let $v_{\mu}=\frac{B_{\mu \bar{\mu}}}{\lambda_{\mu}} \in \mathbb{R}$. Now we compute the expression in \ref{convexeq}. 
\begin{align}
\frac{\partial^2 F}{\partial A_{\mu \bar{\nu}} A_{\alpha \bar{\beta}}} B_{\mu \bar{\nu}}B_{\alpha \bar{\beta}} &=\frac{1}{\det(A)}\displaystyle \sum_{\mu, \alpha} \Bigg (-c_t v_{\alpha}(\lambda_{\mu}+\lambda_{\alpha}) v_{\mu} \nonumber \\ 
&+ \left (c_t \sum_i \lambda_i +2t\tan(\hth)\right ) \left (\frac{\vert B_{\mu \bar{\alpha}} \vert^2}{\lambda_{\mu}\lambda_{\alpha}}+v_{\mu} v_{\alpha} \right )\Bigg ).
\end{align}
Therefore,
\begin{align}
\frac{\partial^2 F}{\partial A_{\mu \bar{\nu}} A_{\alpha \bar{\beta}}} B_{\mu \bar{\nu}}B_{\alpha \bar{\beta}}&\geq \frac{1}{\det(A)}\displaystyle \sum_{\mu, \alpha} \Bigg (-c_t v_{\alpha}(\lambda_{\mu}+\lambda_{\alpha}) v_{\mu} 
+ \left (c_t \sum_i \lambda_i +2t\tan(\hth)\right ) ( v_{\mu} v_{\alpha} +v_{\mu}^2)\Bigg ) \nonumber \\
&+\sum_{\mu \neq \alpha}\frac{c_t \sum_i \lambda_i +2t\tan(\hth)}{\det(A)} \frac{\vert B_{\mu \bar{\alpha}} \vert^2}{\lambda_{\mu}\lambda_{\alpha}} .
\label{computationofthestrongconvexexpression}
\end{align}
Using \ref{conditiononB} in \ref{computationofthestrongconvexexpression} we get the following inequality.
\begin{gather}
\frac{\partial^2 F}{\partial A_{\mu \bar{\nu}} A_{\alpha \bar{\beta}}} B_{\mu \bar{\nu}}B_{\alpha \bar{\beta}}  \geq \frac{1}{\det(A)}\displaystyle \left (-\sum_{\mu, \alpha}c_t v_{\mu} \lambda_{\alpha}v_{\alpha} + \left (c_t \sum_i \lambda_i +2t\tan(\hth)\right )\sum_{\mu}v_{\mu}^2 \right ) \nonumber \\
+\sum_{\mu \neq \alpha}\frac{c_t \sum_i \lambda_i +2t\tan(\hth)}{\det(A)} \frac{\vert B_{\mu \bar{\alpha}} \vert^2}{\lambda_{\mu}\lambda_{\alpha}}.
\label{convexitysimplifiedabit}
\end{gather}
Assume that $\lambda_3 \leq \lambda_2 \leq \lambda_1$. Solving for $v_3$ from \ref{conditiononB} and substituting in \ref{convexitysimplifiedabit} we get the following inequality.
\begin{gather}
\frac{\partial^2 F}{\partial A_{\mu \bar{\nu}} A_{\alpha \bar{\beta}}} B_{\mu \bar{\nu}}B_{\alpha \bar{\beta}} \geq\frac{Ev_1^2+Dv_2^2 +2Bv_1 v_2}{\det(A)(c_t(\lambda_1+\lambda_2)+2t\tan(\hth))} +\sum_{\mu \neq \alpha}\frac{c_t \sum_i \lambda_i +2t\tan(\hth)}{\det(A)} \frac{\vert B_{\mu \bar{\alpha}} \vert^2}{\lambda_{\mu}\lambda_{\alpha}},
\label{finalsimplificationofconvexity}
\end{gather}
where
\begin{align}
E&= 2\left(c_t (\lambda_2+\lambda_3)+2t\tan(\hth)\right) \left(c_t\left(\sum_i \lambda_i\right)+ +2t\tan(\hth) \right) >0\nonumber \\
D&= 2\left(c_t (\lambda_1+\lambda_3)+2t\tan(\hth)\right) \left(c_t\left(\sum_i \lambda_i\right)+ +2t\tan(\hth)\right) >0\nonumber \\
B&= 2 \left (c_t \lambda_3+t\tan(\hth)\right )  \left(c_t\left(\sum_i \lambda_i\right)+ +2t\tan(\hth) \right).
\label{valuesofADB}
\end{align}
Therefore, all we need to do is to prove that $ED-B^2 >0$. Upon computation we get the following.
\begin{align}
ED-B^2&=4\left(c_t\sum_i \lambda_i +2t\tan(\hth) \right )^2\nonumber \\
&\times \Bigg \{\left(c_t (\lambda_2+\lambda_3)+2t\tan(\hth)\right)\left(c_t (\lambda_1+\lambda_3)+2t\tan(\hth)\right)
- \left (c_t \lambda_3 +t\tan(\hth) \right )^2 \Bigg \}.
\label{ADminusBsquared}
\end{align}
Equation \ref{ADminusBsquared} can be simplified to the following.
\begin{align}
g&=\frac{ED-B^2}{4\left(c_t\sum_i \lambda_i +2t\tan(\hth) \right)^2}=c_t^2 \displaystyle \sum_{i<j} \lambda_i \lambda_j +2c_t t\tan(\hth)\sum_i \lambda_i +3t^2 \tan^2(\hth) \nonumber \\
&= (c_t \lambda_1 +2t \tan(\hth))(c_t(\lambda_2+\lambda_3)+2t\tan(\hth)) +c_t^2\lambda_2\lambda_3-t^2 \tan^2(\hth).
\label{ADminusBsquaredsimpl}
\end{align}
Since $\lambda_2\lambda_3 \geq c_t\sec^2(\hth)$ and $c_t^3 \sec^2(\hth)> t^2 \tan^2(\hth)$, we see that
\begin{align}
g&> (c_t \lambda_1 +2t \tan(\hth))(c_t(\lambda_2+\lambda_3)+2t\tan(\hth)).
\label{ggeq}
\end{align}
We can minimise $g$ subject to the constraints \ref{lambdasatisfy}. As one approaches the boundary of these constraints, one of the following two possibilities must occur along any convergent subsequence.
\begin{enumerate}
\item $\lambda_2 \lambda_3 \rightarrow c_t \sec^2(\hth)$ and $\lim (c_t(\lambda_2 + \lambda_3)+2t \tan(\hth))>0$: In this case, $\lambda_1\rightarrow \infty$ and hence $g\rightarrow \infty$ by \ref{ggeq}. 
\item $c_t(\lambda_2 + \lambda_3) \rightarrow -2t \tan(\hth)$:  If $\lambda_1 \rightarrow \infty$ then one can see from \ref{ggeq} that $g\geq 0$ in the limit. If $\lambda_1$ stays bounded, then \ref{ggeq} implies that $g\geq 0$ in the limit.
\item $\lambda_3 \rightarrow 0$ : In this case, $\lambda_1, \lambda_2 \rightarrow \infty$ and hence $g \rightarrow \infty$. 
\end{enumerate}
Therefore, if the extrema of $g$ in the interior of the region are $\geq 0$, then $g\geq 0$ on the region. Using Lagrange's multipliers it is easy to see that the local extrema of $g$ occur when $\lambda_1=\lambda_2=\lambda_3=\lambda$. At such a point,
\begin{gather}
g=3c_t^2 \lambda^2 +6c_t t\tan(\hth) \lambda + 3t^2 \tan^2 (\hth) = (\sqrt{3}c_t \lambda+\sqrt{3}t\tan(\hth))^2 \geq 0.
\end{gather}
Hence, 
\begin{gather}
\frac{\partial^2 F}{\partial A_{\mu \bar{\nu}} A_{\alpha \bar{\beta}}} B_{\mu \bar{\nu}}B_{\alpha \bar{\beta}}\geq \frac{c_t \sum_i \lambda_i +2t\tan(\hth)}{\det(A)} \sum_{\alpha \neq \mu}\frac{\vert B_{\mu \bar{\alpha}} \vert^2}{\lambda_{\mu}\lambda_{\alpha}}.
\end{gather}
\end{proof}

\begin{proposition}
Along \ref{continuitypath}, $\Vert \Delta \phi \Vert_{C^0} \leq C_2(1+\Vert \nabla \phi \Vert_{C^0}^2)$ for some $C_2$ independent of $t$.
\label{laplacianestimateprop}
\end{proposition}
\begin{proof}
Suppose not. Then there is a sequence of times $t_n \rightarrow t$ such that $$\Vert \Delta \phi_{n} \Vert_{C^0} \geq n (1+\Vert \nabla \phi_n \Vert_{C^0}^2).$$ Since $\Delta \phi_n \geq -3$, we see that $\max \Delta \phi_n \rightarrow \infty$. Denote the eigenvalues of the endomorphism $ A_{i,n} ^j (p)= (\Omega_{\phi_n})_{i\bar{k}}\omega^{i\bar{k}} (p)$ as $\lambda_{1,n}(p) \geq \lambda_{2,n}(p) \geq \lambda_{3,n}(p)$. Let $f_n$ be any sequence of functions that are uniformly bounded independent of $n$. Suppose the maximum of the function $\hat{\psi}_n = f_n + \frac{1}{2} \ln (1+\lambda_{1,n}^2)$ is attained at $p_n$. If $\lambda_{1,n} (p_n)$ remains bounded (independent of $n$), then so does $\hat{\psi}_n (p_n)$ and hence so does $\lambda_{1,n}$. However, $3\lambda_{1,n} \geq \max (\Delta \phi_n) \rightarrow \infty$. Therefore, 
\begin{gather}
\lambda_{1,n}(p_n) \rightarrow \infty.
\label{lamba1infty}
\end{gather}
We suppress the point $p_n$ for the remainder of this discussion. Since $$\lambda_{3,n}\lambda_{2,n} = \lambda_{2,n}\sec^2(\hth)\frac{c_{t_n}(\lambda_{1,n}+\lambda_{2,n})+2t_n\tan(\hth)}{\lambda_{1,n}\lambda_{2,n}-c_{t,n}\sec^2(\hth)}$$ and $$\lambda_{2,n}^2 \geq \lambda_{2,n}\lambda_{3,n}\geq c_{t_n} \sec^2(\hth),$$ we see that 
\begin{gather}
c_t \sec^2(\hth) \leq \liminf \lambda_{2,n} \lambda_{3,n}  \leq \limsup \lambda_{2,n} \lambda_{3,n} \leq 2 c_t \sec^2(\hth).
\label{lambda2lambda3}
\end{gather}
We drop the  subscript $n$ in what follows. Unless specified otherwise, the constant $C$ changes from line to line for the remainder of the paper.\\
\indent Define the function $$\psi=-\gamma(\phi)+\frac{1}{2} \ln (1+\lambda_1^2),$$ where $\gamma$ is a function defined as follows. 
\begin{gather}
\gamma (t) = 2A(t+C_0) - \frac{A\tau}{2}(t+C_0)^2, \ -C_0 \leq t \leq C_0.
\label{defofgammandpsi}
\end{gather}
where  $A>>1, \frac{2}{C_0}> \tau>0$ are constants to be determined later. In particular, $\tau$ is chosen to be small enough so that $\gamma>0, \gamma' \in [A,2A]$. Note that our $\psi$ is quite similar to (but simpler than) the corresponding function in the proof of Theorem $4.1$ in \cite{collinsjacobyau}. However, the estimates obtained will be somewhat different. \\
\indent Clearly $\psi$ is a continuous function and hence attains a maximum at a point $p$. Choose normal coordinates for $\omega$ at $p$ and furthermore, choose them to diagonalise $\Omega_{\phi}$ to $\Omega_{\phi}= \displaystyle \sum_i \lambda_i \sqrt{-1}dz^i \wedge d\bar{z}^i$, where $\lambda_3 \leq \lambda_2 \leq \lambda_1$. Near $p$, $\lambda_1$ need not vary smoothly. The standard remedy for this problem is to consider a small perturbation, i.e., a constant diagonal matrix $B$ such that $0=B_{11}>B_{22}>B_{33}$. Define a new matrix $$\tilde{A}^i_{j}=(\Omega_{\phi})_{j\bar{k}}\omega^{i\bar{k}}-B.$$
$B$ is chosen to be so small that $\tilde{A}>0.95A>0$ near $p$. Clearly the largest eigenvalue $\tilde{\lambda}_1$ of $\tilde{A}$ varies smoothly near $p$, and $\tilde{\psi}=-\gamma(\phi)+\frac{1}{2}\ln(1+\tilde{\lambda}_1^2)$ continues to acheive a local maximum at $p$. \\  
Differentiating \ref{continuitypathintermsofF} with respect to $z^i$ we get the following.
\begin{gather}
0=\frac{\partial F}{\partial A_{\mu \bar{\nu}}} ((\Omega_{\phi})_{\mu \bar{\beta}}\omega^{\bar{\beta}\nu})_{,i},
\label{firstderivative}
\end{gather}
where we used the Einstein summation notation. Noting that $\omega_{,i}(p)=0$, we get
\begin{gather}
\frac{\partial F}{\partial A_{\mu \bar{\nu}}} (\Omega_{\phi})_{\mu \bar{\nu},i}=0, \nonumber
\end{gather}
which implies
\begin{align}
 \displaystyle \frac{\partial}{\partial A_{\mu \bar{\nu}}}\vert_{[A]=diag(\lambda_i)} \left(c_t\frac{\mathrm{tr}(A)}{\det(A)}+2t \tan(\hth) \frac{1}{\det(A)}\right )(\Omega_{\phi})_{\mu \bar{\nu},i} &= 0 \nonumber \\
\Rightarrow \frac{1}{\det(A)}\left (c_t-\frac{c_t \sum_i \lambda_i +2t\tan(\hth)}{\lambda_{\mu}} \right ) (\Omega_{\phi})_{\mu \bar{\mu},i}&=0.
\label{simplifiedfirstderivativeatp}
\end{align}
\begin{remark}
Note that the term in the parantheses in the last equation of \ref{simplifiedfirstderivativeatp} is positive because ellipticity is preserved along the continuity path.
\label{aremarkaboutellipticity}
\end{remark}
 Returning back to Equation \ref{firstderivative}, differentiating it again at $p$ with respect to $\bar{z}^i$, and summing over $i$, we get the following in normal coordinates.
\begin{align}
0&=\sum_{i}\left (\sum_{\mu}\frac{\partial F}{\partial \lambda_{\mu}}\lambda_{\mu} R^{\mu \bar{\nu}}_{i\bar{i}} +\sum_{\mu}\frac{\partial F}{\partial \lambda_{\mu}} (\Omega_{\phi})_{\mu \bar{\mu},i\bar{i}}+\frac{\partial^2 F}{\partial A_{\mu \bar{\nu}} A_{\alpha \bar{\beta}}} (\Omega_{\phi})_{\mu \bar{\nu},i}(\Omega_{\phi})_{\alpha \bar{\beta},\bar{i}} \right).
\label{secondderivativesimplified}
\end{align}
Using Lemma \ref{convexitylem} in \ref{secondderivativesimplified} we get the following.
\begin{gather}
0\geq \displaystyle \sum_{i,\mu}\Bigg ( \frac{\partial F}{\partial \lambda_{\mu}}\lambda_{\mu} R^{\mu \bar{\mu}}_{i\bar{i}} +\frac{\partial F}{\partial \lambda_{\mu}} (\Omega_{\phi})_{\mu \bar{\mu},i\bar{i}}\Bigg ) + \frac{c_t \sum_j \lambda_j +2t\tan(\hth)}{\det(A)} \sum_i \sum_{\alpha \neq \mu}\frac{\vert (\Omega_{\phi})_{\mu \bar{\alpha},i} \vert^2}{\lambda_{\mu}\lambda_{\alpha}} \nonumber \\
\Rightarrow \sum_{i,\mu} \left (-\frac{\partial F}{\partial \lambda_{\mu}} \right)  (\Omega_{\phi})_{\mu \bar{\mu},i\bar{i}} \geq  \sum_{i,\mu} \frac{\partial F}{\partial \lambda_{\mu}}\lambda_{\mu} R^{\mu \bar{\mu}}_{i\bar{i}}  + \frac{c_t \sum_j \lambda_j +2t\tan(\hth)}{\det(A)} \sum_i \sum_{\alpha \neq \mu}\frac{\vert (\Omega_{\phi})_{\mu \bar{\alpha},i} \vert^2}{\lambda_{\mu}\lambda_{\alpha}}.
\label{secondderivfinal}
\end{gather}
At the local maximum $p$ of $\tilde{\psi}$, the following hold.
\begin{align}
0=\tilde{\psi}_{,\mu}&= -\gamma' \phi_{,\mu}+\frac{\tilde{\lambda}_1}{1+\tilde{\lambda}_1 ^2} (\Omega_{\phi})_{1\bar{1},\mu}, \ \mathrm{and}\nonumber \\
0\geq \sum_{\mu} \left ( -\frac{\partial F}{\partial \lambda_{\mu}}(A)\right )\tilde{\psi}_{,\mu \bar{\mu}} &= \sum_{\mu} \left ( -\frac{\partial F}{\partial \lambda_{\mu}}\right ) \Bigg ( \left ( \frac{\tilde{\lambda}_1}{1+\tilde{\lambda}_1 ^2} \tilde{\lambda}_{1,\mu}\right )_{,\bar{\mu}} -(\gamma' \phi_{,\mu})_{,\bar{\mu}}\Bigg )\nonumber \\
&\geq A+B.
\label{conditionsonthederivativesofpsiatthemaximum}
\end{align}
Now we calculate $A$.
\begin{align}
A&= \sum_{\mu} \left ( -\frac{\partial F}{\partial \lambda_{\mu}}\right ) \Bigg (\frac{1-\tilde{\lambda}_1^2}{(1+\tilde{\lambda}_1^2)^2}\vert \tilde{\lambda}_{1,\mu} \vert^2 +  \frac{\tilde{\lambda}_1}{1+\tilde{\lambda}_1 ^2} \tilde{\lambda}_{1,\mu \bar{\mu}}\Bigg ) \nonumber \\
&= \sum_{\mu} \left ( -\frac{\partial F}{\partial \lambda_{\mu}}\right ) \Bigg (\frac{1-\tilde{\lambda}_1^2}{(1+\tilde{\lambda}_1^2)^2}\vert (\Omega_{\phi})_{1\bar{1},\mu} \vert^2 \nonumber \\
&+  \frac{\tilde{\lambda}_1}{1+\tilde{\lambda}_1 ^2}  \Bigg [R^{1\bar{1}}_{\mu \bar{\mu}}\lambda_1+(\Omega_{\phi})_{\mu \bar{\mu},1\bar{1}}+\sum_{q>1}\frac{\vert(\Omega_{\phi})_{q\bar{1},\mu}\vert^2+\vert(\Omega_{\phi})_{1\bar{q},\mu}\vert^2}{\lambda_1-\tilde{\lambda}_q} \Bigg ]\Bigg ) \nonumber \\
&\geq \sum_{\mu} \left ( -\frac{\partial F}{\partial \lambda_{\mu}}\right ) \Bigg (\frac{1-\lambda_1^2}{(1+\lambda_1^2)^2}\vert (\Omega_{\phi})_{1\bar{1},\mu} \vert^2 
-  C + \frac{\lambda_1}{1+\lambda_1 ^2}(\Omega_{\phi})_{\mu \bar{\mu},1\bar{1}} \Bigg ),
\label{equationforA}
\end{align}
where the second-to-last line follows from Equation 4.4 in \cite{collinsjacobyau}. Now we use \ref{secondderivfinal} in \ref{equationforA} to get the following.
\begin{align}
A&\geq \sum_{\mu} \left ( -\frac{\partial F}{\partial \lambda_{\mu}}\right ) \Bigg (\frac{1-\lambda_1^2}{(1+\lambda_1^2)^2}\vert (\Omega_{\phi})_{1\bar{1},\mu} \vert^2 
-  C -\frac{C\lambda_1\lambda_{\mu}}{1+\lambda_1 ^2}\Bigg )\nonumber \\
&+ \frac{c_t \sum_j \lambda_j +2t\tan(\hth)}{(1+\lambda_1^2)\lambda_1\lambda_2 \lambda_3}  \sum_{1< \mu}\frac{\vert (\Omega_{\phi})_{\mu \bar{1},1} \vert^2}{\lambda_{\mu}} \nonumber \\
&\geq \left ( -\frac{\partial F}{\partial \lambda_{1}}\right ) \Bigg (-(\gamma')^2 \vert \phi_{,1} \vert^2 -  C \Bigg ) + \sum_{\mu>1} \Bigg [\left ( -\frac{\partial F}{\partial \lambda_{\mu}}\right ) \left ( 
-  C -\frac{C\lambda_1\lambda_{\mu}}{1+\lambda_1 ^2}\right ) \nonumber \\
&+ \frac{\vert (\Omega_{\phi})_{1 \bar{1},\bar{\mu}} \vert^2}{1+\lambda_1^2}\left (\left ( -\frac{\partial F}{\partial \lambda_{\mu}}\right ) \frac{1-\lambda_1^2}{1+\lambda_1^2}+ \frac{c_t \sum_j \lambda_j +2t\tan(\hth)}{\lambda_{\mu}\lambda_1\lambda_2 \lambda_3} \right )  \Bigg ].
\label{simplifiedinequalityonthesecondderivativeofpsiatthemaximum}
\end{align}
Using the expression for $\frac{\partial F}{\lambda_{\mu}}$ from \ref{simplifiedfirstderivativeatp} we get the following.
\begin{align}
A&\geq \left ( -\frac{\partial F}{\partial \lambda_{1}}\right ) \Bigg (-(\gamma')^2 \vert \phi_{,1} \vert^2 -  C \Bigg ) + \sum_{\mu>1} \Bigg [\left ( -\frac{\partial F}{\partial \lambda_{\mu}}\right ) \left ( 
-  C -\frac{C\lambda_1\lambda_{\mu}}{1+\lambda_1 ^2}\right ) \nonumber \\
&+ \frac{\vert (\Omega_{\phi})_{1 \bar{1},\bar{\mu}} \vert^2}{(1+\lambda_1^2)\lambda_1 \lambda_2 \lambda_3}\left (c_t\frac{\lambda_1^2-1}{\lambda_1^2+1}+ \frac{2 \lambda_1 ^2 (c_t \sum_j \lambda_j +2t\tan(\hth))}{\lambda_{\mu}(1+\lambda_1^2)} \right )  \Bigg ] \nonumber \\
&\geq \left ( -\frac{\partial F}{\partial \lambda_{1}}\right ) \Bigg (-(\gamma')^2 \vert \phi_{,1} \vert^2 -  C \Bigg ) + \sum_{\mu>1} \left ( -\frac{\partial F}{\partial \lambda_{\mu}}\right ) \left ( 
-  C -\frac{C\lambda_1\lambda_{\mu}}{1+\lambda_1 ^2}\right ).
\label{furthersimplifiedinequalityonthesecondderivativeofpsiatthemaximum}
\end{align}
At this juncture, we calculate a lower bound on $\frac{\partial F}{\partial \lambda_{\mu}}\lambda_{\mu}$.
\begin{align}
\frac{\partial F}{\partial \lambda_{\mu}}\lambda_{\mu} &= \frac{c_t \lambda_{\mu} - \sum_i c_t \lambda_i -2t\tan(\hth)}{\lambda_1 \lambda_2 \lambda_3} \nonumber \\
&=-\frac{\sum_{i\neq \mu} c_t \lambda_i +2t\tan(\hth)}{\sec^2(\hth) \left(c_t\sum_i\lambda_i+2t\tan(\hth) \right )} \nonumber \\
&\geq - \cos^2(\hth).
\label{lowerboundonEuler}
\end{align}
Using \ref{lowerboundonEuler} in \ref{furthersimplifiedinequalityonthesecondderivativeofpsiatthemaximum} we get the following.
\begin{align}
A&\geq \frac{\partial F}{\partial \lambda_{1}}(\gamma')^2 \vert \phi_{,1} \vert^2 +C \left(\sum_{\mu} \frac{\partial F}{\partial \lambda_{\mu}} - \frac{1}{\lambda_1} \right).
\label{finalsymplificationonA} 
\end{align}
We now shift our attention to $B$.
\begin{align}
B &= \sum_{\mu}\frac{\partial F}{\partial \lambda_{\mu}} \left (\gamma'' \vert \phi_{,\mu} \vert^2+ \gamma' \phi_{,\mu \bar{\mu}}\right ).
\end{align}
Noting that $\gamma'' = -A\tau<0$, we see that
\begin{align}
B &\geq \sum_{\mu} \frac{\partial F}{\partial \lambda_{\mu}} \gamma' \phi_{,\mu \bar{\mu}}.
\label{simplifiedinequalityforB}
\end{align}
At this point we need the following lemma which is standard in these sorts of problems.
\begin{lemma}
There exists a uniform $N>0$ such that if $\lambda_1 >N$, then there exists a uniform constant $\kappa>0$ such that 
\begin{gather}
\sum_{\mu} \frac{\partial F}{\partial \lambda_{\mu}} \phi_{,\mu \bar{\mu}} \geq \kappa \sum_{\mu} -\frac{\partial F}{\partial \lambda_{\mu}}.
\label{theweisunlemmaequation}
\end{gather}
\label{theweisunlemma}
\end{lemma}
\begin{proof}
Choose $\kappa>0$ to be small enough so that $\Omega-\kappa \omega>\kappa \omega>0$.
\begin{gather}
(\Omega-\kappa \omega)^2 > c_t\sec^2(\hth) \omega^2.
\label{kappaOmega}
\end{gather}
Note that $\phi_{,\mu\bar{\mu}}=\lambda_{\mu}-\Omega_{\mu \bar{\mu}}$, and that $\lambda_1 \rightarrow \infty$. Also note that
\begin{gather}
\sum_{\mu} \frac{\partial F}{\partial \lambda_{\mu}} (\phi_{,\mu \bar{\mu}}+\kappa)=-\sum_{\mu}\frac{(\sum_{i\neq \mu}c_t\lambda_i +2t\tan(\hth)) (\lambda_{\mu}+\kappa-\Omega_{\mu \bar{\mu}}	)}{\lambda_{\mu}\sec^2(\hth)(c_t\sum_i \lambda_i+2t\tan(\hth))}.
\label{eq:useful}
\end{gather}
For every convergent subsequence, there are two possibilites (we abuse notation and denote the limits of $\lambda_i$ by the same letter):
\begin{enumerate}
\item $\lambda_2$ stays bounded:  We see that 
\begin{align}
\sum_{\mu} \frac{\partial F}{\partial \lambda_{\mu}} (\phi_{,\mu \bar{\mu}}+\kappa) &\rightarrow  -\cos^2(\hth)\frac{ (\lambda_2+\kappa-\Omega_{2\bar{2}})}{\lambda_2} - \cos^2(\hth)\frac{c_t (\lambda_3+\kappa-\Omega_{3\bar{3}})}{\lambda_3} \nonumber \\
&=\cos^2(\hth)\left (-2+\frac{\Omega_{2\bar{2}}-\kappa}{\lambda_2}+ \frac{\Omega_{3\bar{3}}-\kappa}{\lambda_3} \right ) \nonumber \\
&\geq 2\cos^2(\hth) \left ( -1 + \sqrt{\frac{(\Omega_{2\bar{2}}-\kappa)(\Omega_{3\bar{3}}-\kappa)}{\lambda_2\lambda_3}} \right )
\label{prelimcalcoflhsinweisun}
\end{align}
Using \ref{kappaOmega} in \ref{prelimcalcoflhsinweisun}, if $\lambda_1$ is sufficiently large, 
\begin{align}
\sum_{\mu} \frac{\partial F}{\partial \lambda_{\mu}} (\phi_{,\mu \bar{\mu}}+\kappa) &> \cos^2(\hth)(-1+1)=0,
\end{align}
\item $\lambda_2\rightarrow \infty$: In this case, the expression for $\lambda_3$ shows that it goes to $0$. The first two terms of \ref{eq:useful} are bounded above whereas the third goes to $\infty$. 
\end{enumerate}
In either case, we are done.
\end{proof}
Using \ref{conditionsonthederivativesofpsiatthemaximum}, \ref{finalsymplificationonA}, \ref{simplifiedinequalityforB}, and Lemma \ref{theweisunlemma}, we obtain the following inequality.
\begin{align}
0&\geq \frac{\partial F}{\partial \lambda_{1}}(\gamma')^2 \vert \phi_{,1} \vert^2 +C \left(\sum_{\mu} \frac{\partial F}{\partial \lambda_{\mu}} - \frac{1}{\lambda_1} \right)+ \gamma' \kappa\sum_{\mu} -\frac{\partial F}{\partial \lambda_{\mu}}\nonumber \\
&\geq \frac{\partial F}{\partial \lambda_{1}}4A^2 \vert \phi_{,1} \vert^2  - \frac{C}{\lambda_1} + \frac{A \kappa}{2}\sum_{\mu} -\frac{\partial F}{\partial \lambda_{\mu}},
\end{align}
where we chose $A$ to be sufficiently large and used the fact that $2A\geq\gamma' \geq A$. Therefore,
\begin{align}
0&\geq -\frac{4A^2\vert \phi_{,1} \vert^2 \cos^2(\hth)}{\lambda_1}\frac{c_t(\lambda_2+\lambda_3)+2t\tan(\hth)}{c_t \sum_i \lambda_i + 2t\tan(\hth)}-\frac{C}{\lambda_1} \nonumber \\
&+ \frac{A\kappa \cos^2(\hth)}{2}\sum_i \frac{c_t\sum_{j\neq i}\lambda_j+2t\tan(\hth)}{\lambda_i(c_t\sum_{k} \lambda_k+2t\tan(\hth))} \nonumber\\
\Rightarrow \frac{C(1+\vert \phi_{,1} \vert^2)}{\lambda_1}&\geq \frac{A\kappa \cos^2(\hth)}{\sqrt{\lambda_2\lambda_3}}\left(\frac{(c_t(\lambda_1+\lambda_3)+2t\tan(\hth))(c_t(\lambda_1+\lambda_2)+2t\tan(\hth))}{(c_t\sum_k \lambda_k +2t\tan(\hth))^2} \right)^{1/2}.
\label{eq:almostfinal}
\end{align} 
Using \ref{lambda2lambda3} we see that the right-hand-side of \ref{eq:almostfinal} is bounded below. Therefore, at the maximum of $\psi$, $$\lambda_1 \leq C(1+\max \vert \nabla \phi \vert^2).$$ Since $\vert \phi \vert$ is bounded, this bound  translates into the desired bound on the Laplacian of $\phi$.
\end{proof}
At this juncture, Proposition 5.1 of \cite{collinsjacobyau} (also see the blow-up argument in \cite{dinewkol}) shows that $$\Vert \Delta \phi \Vert_{C^0} \leq \tilde{C}_2,$$
where $\tilde{C}_2$ is independent of $t$. It is also easily to see that this Laplacian bound implies a bound on the complex Hessian $\spbp \phi$ and that the operator $F$ is uniformly elliptic.
\section{Higher order estimates and conclusion of the proof}\label{concludingsec}
\indent Since $F$ is convex only when restricted to a level set, the (complex version of) standard Evans-Krylov theorem is not applicable directly. Typically, one uses an auxilliary function $G$ such that $G\circ F$ satisfies the conditions of the Evans-Krylov theorem \cite{collinsjacobyau}. However, since it is not clear whether convexity and ellipticity hold on general level sets, we need a slightly different technique to prove $C^{2,\alpha}$ estimates.
\begin{proposition}
Along the continuity path \ref{continuitypathintermsofF}, $\Vert \phi \Vert_{C^{2,\alpha}}\leq C_3$ where $C_3$ is independent of $t$.
\label{EvansKrylovprop}
\end{proposition}
\begin{proof}
As earlier, unless specified otherwise, the constant $C>0$ is allowed to vary from line to line. The manifold $X$ can be covered with finitely many relatively compact coordinate open charts $\bar{U}_{\gamma}\subset V_{\gamma}$ such that $\Omega_{\phi} =\spbp u_{\gamma}$ on $V_{\gamma}$ for a smooth function $u_{\gamma}$ where $\Vert u_{\gamma} \Vert_{C^0(\bar{U}_\gamma)}+\Vert \Delta_{euc} u \Vert_{C^0(\bar{U}_{\gamma})}\leq C$ for some constant $C$ independent of $t$. We drop the subscript $\gamma$ for the remainder of this proposition and focus our attention to a coordinate ball $B_2$ in a  coordinate chart.\\
\indent The function $u$ satisfies 
\begin{gather}
F(A)=\sec^2(\hth) \ \mathrm{where} \ A^i_j = \omega^{i\bar{k}}u_{,i\bar{k}}.
\label{equationlocally}
\end{gather}
At this point, assuming that the desired estimate does not hold,  a standard blow-up argument (the proof of Lemma 6.2 in \cite{collinsjacobyau}) can be used to produce a point $x_0 \in B_1$ and a $C^{3,\alpha}$ function $u :\mathbb{C}^n \rightarrow \mathbb{R}$ satisfying
\begin{gather}
F_{x_0}(A)=\sec^2(\hth), \ \vert \partial \bar{\partial} \partial u \vert (0)=1, \ \mathrm{and}, \nonumber \\
\Vert \spbp u \Vert_{L^{\infty}(\mathbb{C}^n)} \leq C <\infty,
\label{localequation}
\end{gather}
where $F_{x_0}(A) = F(\omega(x_0)^{i\bar{k}}u_{,i\bar{k}})$ is a constant-coefficient nonlinear elliptic operator. By elliptic regularity, $u$ is actually smooth. We now prove a Liouville-type result on $u$.
\begin{lemma}
A smooth function $u :\mathbb{C}^n \rightarrow \mathbb{R}$ satisfying \ref{localequation} is a quadratic polynomial.
\label{Liouville}
\end{lemma}
\begin{proof}
This result is a straightforward corollary of the following proposition.
\begin{proposition}
There exists a constant $\beta>0$ such that for every $R>0$, the function $u$ satisfies $\Vert \spbp u \Vert_{C^{\beta}(B_R(0))}\leq CR^{-\beta}$.
\label{evanskrylovconstant}
\end{proposition}
\begin{proof}
We follow the exposition of the proof of the complex version of the Evans-Krylov theorem in \cite{Siu}. Differentiating Equation \ref{localequation} twice in the direction $\partial_{\gamma}$ and using the convexity property in Lemma \ref{convexitylem}, we get the following.
\begin{gather}
\left (-\frac{\partial F_{x_0}}{\partial A_{i\bar{j}}} \right )u_{,i\bar{j}\gamma\bar{\gamma}} \geq 0. 
\label{difflocaleqtwiceanduseconvex}
\end{gather}
Let $w=u_{,\gamma\bar{\gamma}}$ and $M_R, m_R$ be the maximum and minimum values of $w$ on $B_R(0)$. Then $$g^{i\bar{j}}w_{,i\bar{j}}\geq 0$$ and hence the weak Harnack inequality implies that
\begin{gather}
\left ( \frac{\int_{B_R(0)}(M_{2R}-w)^p}{R^n} \right )^{1/p} \leq C (M_{2R}-M_{R}).
\label{weakHarnack}
\end{gather}
To get a H\"older estimate on $w$, one needs a similar bound on $-w$. To this end, consider the following.
\begin{align}
0&=F_{x_0}(A(y))-F_{x_0}(A(x)) \nonumber \\
&=\displaystyle (u_{i\bar{j}}(x)-u_{i\bar{j}}(y))\int_0^1 \left( -\frac{\partial F_{x_0}}{\partial A_{i\bar{j}}}(sA(x)+(1-s)A(y)) \right) ds \nonumber \\
&=\displaystyle (u_{i\bar{j}}(x)-u_{i\bar{j}}(y)) K^{i\bar{j}}
\end{align}
We now bound the eigenvalues of the matrix $K^{i\bar{j}}$ above and below. Let $v$ be a unit eigenvector with eigenvalue $\lambda$.
\begin{align}
\lambda &= v_i\bar{v}_j K^{i\bar{j}}=\displaystyle \int_0 ^1 v_i \bar{v}_j \left(-\frac{\partial F_{x_0}}{\partial A_{i\bar{j}}}(sA(x)+(1-s)A(y)) \right )ds\nonumber \\
&=\int_0 ^1 \frac{\vert \tilde{v}_{\mu} \vert^2}{\det(sA(x)+(1-s)A(y))}\left ( \frac{c_t \mathrm{tr}(sA(x)+(1-s)A(y))+2t\tan(\hth)}{\lambda_{\mu}(sA(x)+(1-s)A(y))}-c_t\right ) ds, 
\label{eigenvaluebound}
\end{align}
where the last equality is obtained by diagonalisation (which varies with $s$). It is well known (see \cite{Lieb} for instance) that   $\lambda_2(A)+\lambda_3(A)$ is a concave function of $A$. Therefore, the following holds for $i\neq j$.
\begin{align}
c_t(\lambda_i+\lambda_j)(sA(x)+(1-s)A(y)) &\geq c_t(\lambda_2+\lambda_3)(sA(x)+(1-s)A(y)) \nonumber \\
&\geq s c_t(\lambda_2+\lambda_3)(A(x)) +(1-s)c_t(\lambda_2+\lambda_3)(A(y)) \nonumber \\
&> s2t\tan(\hth)+(1-s)2t\tan(\hth)=2t\tan(\hth).
\label{keyconcavity}
\end{align}
Since positive-definiteness is preserved under convex combinations, \ref{keyconcavity} and \ref{eigenvaluebound} imply that $\lambda>0$. In fact, since $\lambda_1 (A)$ is a convex function, i.e., $\lambda_1(sA(x)+(1-s)A(y))\leq s\lambda_1 (A(x))+(1-s)\lambda_1(A(y))\leq C$, we see that $\lambda\geq \frac{1}{C}$. Since $\lambda_3(A)$ is concave and $\frac{1}{\det(A)}$ is convex, we see that 
\begin{gather}
\frac{1}{C}\leq \lambda \leq C.
\label{boundsoneigenvalues}
\end{gather}
\indent From this point onwards, the usual proof of the complex version of the Evans-Krylov theorem (in \cite{Siu}) takes over verbatim and we get the necessary H\"older estimate on $\spbp u$.   
\end{proof}
Indeed, taking $R\rightarrow \infty$ in Propositon \ref{evanskrylovconstant}, we are done.
\end{proof}
Lemma \ref{Liouville} shows that $\partial \bar{\partial} \partial u(0)=0$, thus producing a contradiction. Hence, the desired Evans-Krylov type estimate holds.
\end{proof}
Given Proposition \ref{EvansKrylovprop}, we can use the Schauder estimates to bootstrap the regularity to get $C^{k,\alpha}$ \emph{a priori} estimates for all $k$. Using the Arzela-Ascoli theorem we see that closedness holds. This observation completes the proof of Theorem \ref{maintheorem}. \qed

\end{document}